\documentclass[a4paper, 11pt,reqno,oneside]{amsart}
	\usepackage{xcolor}
	\usepackage{amsmath,amssymb,amsthm,mathtools}
	\usepackage[shortlabels]{enumitem}
	\usepackage[left=1in, right=1.5in]{geometry}
	\usepackage[pagebackref]{hyperref}
	\usepackage{cleveref}
	\usepackage{indentfirst}
	\usepackage{comment}
	\usepackage{tikz}
	\usetikzlibrary{graphs,graphs.standard}
	\definecolor{vibrant-grey}{HTML}{BBBBBB}
	\definecolor{vibrant-blue}{HTML}{0077BB}%01
	\definecolor{vibrant-cyan}{HTML}{33BBEE}%02
	\definecolor{vibrant-green}{HTML}{009988}%03
	\definecolor{vibrant-orange}{HTML}{EE7733}%04
	\definecolor{vibrant-red}{HTML}{CC3311}%05
	\definecolor{vibrant-purple}{HTML}{EE3377}%06
	\definecolor{pale-grey}{HTML}{DDDDDD}
	\definecolor{pale-blue}{HTML}{BBCCEE}%01
	\definecolor{pale-cyan}{HTML}{CCEEFF}%02
	\definecolor{pale-green}{HTML}{CCDDAA}%03
	\definecolor{pale-yellow}{HTML}{EEEEBB}%04
	\definecolor{pale-red}{HTML}{FFCCCC}%05

 	 \definecolor{seqColor0}{HTML}{CEFFFF}
	 \definecolor{seqColor1}{HTML}{C6F7D6}
	 \definecolor{seqColor2}{HTML}{A2F49B}
	 \definecolor{seqColor3}{HTML}{BBE453}
	 \definecolor{seqColor4}{HTML}{D5CE04}
	 \definecolor{seqColor5}{HTML}{E7B503}
	 \definecolor{seqColor6}{HTML}{F19903}
	 \definecolor{seqColor7}{HTML}{F6790B}
	 \definecolor{seqColor8}{HTML}{F94902}
	 \definecolor{seqColor9}{HTML}{E40515}
	 \definecolor{seqColor10}{HTML}{A80003}
	 \definecolor{seqColor12}{HTML}{888888}

	\usepackage[textsize=footnotesize, colorinlistoftodos]{todonotes}
	\newtheorem{theorem}{Theorem}[section]
	\newtheorem{lemma}[theorem]{Lemma}
	\crefalias{lemma}{lemma}
	\crefalias{theorem}{theorem}
	\newtheorem{proposition}[theorem]{Proposition}
	\newtheorem{corollary}[theorem]{Corollary}

	\theoremstyle{definition}
	\newtheorem{definition}[theorem]{Definition}

	\theoremstyle{remark}
	\newtheorem{remark}[theorem]{Remark}

	\newcommand{\eps}{\varepsilon}
	\DeclareRobustCommand{\stirling}{\genfrac\{\}{0pt}{}}

	\DeclareMathOperator{\N}{N}
	\DeclareMathOperator{\NR}{NR_4}

	\DeclareMathOperator{\ex}{ex}

	\newcommand{\Rfour}{R_4}

\newcommand{\dRL}{\delta_{\mathrm{rem}}}

\title{Counting edge-colorings of a complete graph avoiding a rainbow $K_4$}
\date{August 25th, 2026}
\author[F. S. Benevides]{Fabrício S. Benevides}
\address[F. S. Benevides]{Departamento de
Matemática, Universidade Federal do Ceará (UFC), Av. Humberto Monte, Campus do Pici (Bloco 914). 60.455-760, Fortaleza, Brazil.}
\email{fabricio@mat.ufc.br}

\author[J. O. Bastos]{Josefran de Oliveira Bastos}
\address[J. O. Bastos]{Universidade Federal do Ceará (UFC), Campus de Sobral, Sobra, Brazil.}
\email{josefran@ufc.br}

\thanks{Research supported by the CAPES-Cofecub project Ma 1004/23, CAPES Finance Code 001. F.\,S.\,Benevides was also supported by CNPq (313558/2025-6), FAPESB (EDITAL FAPESB No 012/2022 - UNIVERSAL - NºAPP0044/2023).}

\begin{document}
\begin{abstract}	
	For $k, r, n$ natural numbers let $\rho_{r,k}(K_n)$ be the number of $r$-edge-colorings of $K_n$ that do not contain a rainbow copy of a $K_k$, that is, a copy of $K_k$ in which all edges receive different colors. When $k=3$, the quantity $\rho_{r,3}(K_n)$ represents the number of Gallai Colorings. It was proved by Balogh and Li and independently by Bastos, Benevides and Han, that most of the Gallai colorings are 2-colorings, for $n$ large.	A natural analogue conjecture would be that when $k=4$, $r\ge 5$ and $n$ large, most rainbow-$K_4$-free $r$-edge-colorings are $5$-colorings. We show that this is not true in general and identify an exact threshold for $r$ where this ceases to be true. For the range where the conjecture is false, we determine the exponential growth of $\rho_{r,4}(K_n)$ for every fixed $r$. More precisely, for \(6\le r\le24\), we prove that \(\rho_{r,4}(K_n)=(\binom{r}{5}+o(1))5^{\binom{n}{2}}\); and for each \(r\ge25\), the proportion using at most five colors tends to zero, and \(\rho_{r,4}(K_n)=r^{(n^2/4)+o(n^2)}\). A bipartite construction, with all edges within the two parts assigned one common color, achieves the latter exponential growth rate. The lower bounds can be easily generalized for every $k$. Those results are related to other recent results about counting colorings that avoid rainbow cliques or given rainbow patterns in general. Our proof combines hypergraph containers with the graph removal lemma, structural estimates for color palettes and a refined count of colorings close to a fixed five-color palette.
\end{abstract}

\maketitle

% \listoftodos

\section{Introduction}
We consider an $r$-coloring of a graph $G$ to be any function $\varphi : E(G) \to [r]$, where $[r] = \{1, \ldots, r\}$. That is, we work only with edge-colorings (and they do not have to be proper colorings). A coloring is called \emph{rainbow-$K_k$-free} if every copy of $K_k$ in $G$ has at least two edges of the same color. Let $\rho_{r,k}(G)$ be the number of rainbow-$K_k$-free $r$-colorings of $G$. We focus on the case where $G=K_n$.

When $k=3$, the value of $\rho_{r,3}(K_n)$ represents the number of \emph{Gallai colorings} and was computed asymptotically in \cite{balogh2019typical} and \cite{2019GallaikCol-JCBT}. Trivially, every $2$-coloring of $K_n$ contains no rainbow $K_3$. These papers independently showed that almost all Gallai colorings of $K_n$, for $n$ large, are $2$-colorings. In \cite{balogh2019typical} the value of $r$ is constant, while in \cite{2019GallaikCol-JCBT} it can grow with $n$, as long as $r = r(n) < 2^{n/4300}$. For $k=4$, since $K_4$ has six edges, a natural conjecture would be that, for $r\ge 5$, almost all rainbow-$K_4$-free $r$-colorings of $K_n$ are $5$-colorings (see for example, the related Problem 8.3 of \cite{gupta2025framework}). We show that this conjecture is false for general $r$ and, as our main contribution, we identify the exact threshold $r=25$ at which this ceases to hold. More precisely, for $5 \le r\le 24$ we show that almost all rainbow-$K_4$-free $r$-colorings of $K_n$ are indeed $5$-colorings. And for every fixed $r\ge 25$ we determine the exponential growth of $\rho_{r,4}(K_n)$.

\begin{theorem}[Main theorem]
	\label{thm:main}
	Let $n$ be a natural number and put $N = \binom{n}{2}$. For every fixed integer \(r\ge1\) and real $\eps >0$, there exists
	\(n_0=n_0(r,\eps)\) such that, for every \(n\ge n_0\), the following statements hold.
	\begin{enumerate}[(i)]
		\item \label{item:main1} If \(r\le5\), then
		      \[	
			      \rho_{r,4}(K_n)=r^N.
		      \]
		      
		\item \label{item:main2} If \(6\le r\le24\), then
		      \[	
			      \left(\!\binom r5-\eps\!\right)5^N
			      \le \rho_{r,4}(K_n) \le \left(\!\binom r5+\eps\!\right)5^N.
		      \]
		      
		\item \label{item:main3} If \(r\ge25\), then 
		      \[	
			      \frac{1}{2}\binom{n}{\lfloor n/2 \rfloor}\!\left( 1 - \frac{1}{r} \right)^{n-1}r^{\lfloor n^2/4\rfloor}
			      \le \rho_{r,4}(K_n) \le r^{(1/4+\eps)n^2}.
		      \]
	\end{enumerate}
\end{theorem}

Our proof uses the container method, as in \cite{balogh2019typical}, together with some structural lemmas that were proved using the graph removal lemma, Szemerédi's regularity lemma and counting arguments. The examples that provide our lower bounds can be easily generalized to give lower bounds for $\rho_{r,k}(K_n)$, for every $k$.

\subsection*{Motivation and related work}

We now discuss the history of estimating $\rho_{r,k}(G)$ and related problems. Our article is at the crossroads of two trends: the study of (i) Gallai colorings and their variants \cite{gyarfas2010gallai, gyarfas2020distribution, 2019GallaikCol-JCBT}, and (ii) counting colorings that avoid subgraphs colored in prescribed ways as in \cite{alon2004number, benevides2017edge, gupta2025framework}. While the former is inspired by a seminal result of Gallai from 1967~\cite{gallai1967transitiv}, the latter is inspired by a classical problem posed by Erd{\H o}s and Rothschild in 1974 \cite{erdos1974some}, who asked for the maximum (over graphs with $n$ vertices) number of $2$-edge-colorings with no \emph{monochromatic} copies of $K_{k}$. 

In the direction of (i), the term \emph{Gallai colorings} was introduced by Gyárfás and Simonyi~\cite{GySi2004edge}, while studying structural properties of such colorings. They noted their connections to other topics such as information theory \cite{korner2000graph} and the (weak) perfect graph theorem~\cite{cameron1986note}. Many Ramsey-type problems (see \cite{gyarfas2010ramsey} for early results and \cite{magnant2020topics} for a survey) have also been studied. The problem of computing $\rho_{r,3}(K_n)$ was settled asymptotically in \cite{balogh2019typical,2019GallaikCol-JCBT},
after several partial results had been obtained \cite{falgas2019multicolor,benevides2017edge}. More recently, the value of $\rho_{3,3}(G)$ has been studied when $G$ is an Erd\H{o}s--R\'enyi random graph~\cite{benevides2026gallai}.

As usual in extremal graph theory, let $\ex(n, H)$ be the maximum number of edges in a graph with $n$ vertices that does not contain $H$ as a subgraph. Turán's theorem (1941) states that $\ex(n, K_{k})$ is the number of edges, $t_{k-1}(n)$, in the Turán graph $T_{k-1}(n)$, with $n$ vertices and $k-1$ parts of sizes differing by at most one.

In the context of (ii), let $F$ be a graph with at least one edge, and let $c_{r,F}(G)$ be the number of $r$-colorings of the edge set of $G$ for which there is no \emph{monochromatic} copy of $F$. The problem consists of finding $c_{r,F}(n)$, the maximum of $c_{r,F}(G)$ over all $n$-vertex graphs $G$. It is clear that 
\begin{eqnarray} \label{ser1}
	c_{r, F}(n) \geq r^{\ex(n,F)}  \textrm{ for every }n\geq 2,
\end{eqnarray}
as any $r$-coloring of the edges of an $F$-extremal $n$-vertex graph is trivially $F$-free. The Erd\H{o}s-Rothschild problem was the case where $F=K_k$ and $r=2$ and they conjectured that, for large $n$, equality is attained in (\ref{ser1}), that is, $c_{2, K_k}(n) = 2^{t_{k-1}(n)}$, and the only extremal graph is $T_{k-1}(n)$. This was proved 
by Yuster (1996) \cite{yuster1996number} for triangles ($k=3$) and by Alon, Balogh, Keevash, and Sudakov (2004) \cite{alon2004number} for general $k$; in the latter paper it was also shown that the analogous statement holds for $r=3$ but fails for $r\ge4$. Pikhurko, Staden, and Yilma (2017)~\cite{pikhurko2017erdHos} expressed the exponential growth rate in terms of a finite optimization problem. Pikhurko and Staden subsequently developed stability results and a framework for obtaining exact solutions under suitable hypotheses~\cite{PikhurkoStaden2023,PikhurkoStaden2024}. Beyond the cases $r=2,3$, exact or asymptotic results are known for only a few pairs $(r,k)$ \cite{alon2004number, pikhurko2012maximum, botler2019maximum, PikhurkoStaden2024}. The function $c_{r,F}(n)$ has been studied for several classes of graphs, such as odd cycles~\cite{alon2004number}, matchings~\cite{hoppen2012edge}, paths and stars~\cite{hoppen2014edge}. The hypergraph analogue of this problem has also been considered~\cite{hoppen2012hypergraphs,hoppen2015edge,LP11,LPS11}. 

Balogh (2006)~\cite{balogh2006remark} introduced a version involving forbidden non-monochromatic colorings of cliques. A version involving forbidden color patterns, rather than prescribed exact color labels, was studied also for cliques by Benevides, Hoppen, and Sampaio (2017)~\cite{benevides2017edge}.

In the same spirit as the original Erd\H{o}s--Rothschild problem, Hoppen, Lefmann, and Odermann \cite{hoppen2017rainbow} asked what graphs, $G$, maximize $\rho_{r,k}(G)$ (changing the focus to avoiding rainbow structures), among all $n$-vertex graphs. They showed that for fixed $k \ge 3$, $r \ge \binom{k}{2}^{8k-4}$ and $n$ large enough, we have $\rho_{r,k}(G) \le r^{\ex(n, K_k)}$, with equality only when $G = T_{k-1}(n)$. Since $T_{k-1}(n)$ does not contain $K_k$, it is clear that any coloring of $T_{k-1}(n)$ does not contain a rainbow copy of $K_k$. The nontrivial assertion is that no $n$-vertex graph has more such colorings, for that range of $r$. 
While the lower bound $\binom{k}{2}^{8k-4}$ on $r$ is not optimal, the authors observe that the result does not hold when $r$ is small compared to $k$. 
When $k=4$, this bound was improved in \cite{hoppen2021extension} to $r\ge 5434$ and finally in  \cite{han2025maximumK4} to $r\ge 12$, which is best possible (see Problem 8.3 in \cite{gupta2025framework}). Finding the extremal graph for this problem when $6 \le r \le 11$ remains an open problem. A general framework to tackle this kind of problem is given in \cite{gupta2025framework}, where the reader can also find a table of known extremal configurations and counting results.

While studying $\max\{ \rho_{r,k}(G) : |V(G)|=n\}$ is closely related to the original Erd{\H o}s-Rothschild problem, the direct question of estimating $\rho_{r,k}(K_n)$ is just as natural, interesting and closer to the problem of counting Gallai colorings (even if $K_n$ is not a graph that attains the maximum). Here, we have characterized, asymptotically, the possible values of $\log_r \rho_{r,4}(K_n)$ for \emph{every} fixed $r$.

\subsection{Proof of lower bounds in \texorpdfstring{\Cref{thm:main}}{main theorem} (extremal constructions)} \label{sec:lower}

Let $N=\binom{n}{2}$. When $r \le 5$ the result (\ref{item:main1}) is trivial, since a rainbow copy of $K_4$ requires six different colors. Assume $r \ge 6$. Of course we still have $\rho_{r,4}(K_n) \ge 5^{\binom{n}{2}}$, since we can fix a set $C\subseteq[r]$ of five colors and use only those colors. We would like to multiply this bound by $\binom{r}{5}$; however, we would be counting certain colorings multiple times and we would need to use inclusion-exclusion to compute the exact number of $5$-colorings using colors from $[r]$. One way to simplify this is to count only the colorings in which all colors in $C$ are used at least once. For a fixed $C$, the number of such colorings is the number of surjective maps from $\left[N \right]$ to $[5]$. This is equal to $5!\cdot \stirling{N}{5}$, where $\stirling{N}{5}$ denotes the Stirling number of the second kind. For $N$ large, the ratio $5!\,\stirling{N}{5}/5^N$ approaches 1. Therefore,
\begin{equation}	
	\label{main:eq1}
	\rho_{r,4}(K_n) \ge \big(1-o(1)\big) \binom{r}{5}\,5^{\binom{n}{2}}, \text{ for all } r \ge 5.
\end{equation}
(In fact, it would be enough to use that $5!\,\stirling{N}{5} \ge 5^N-5\cdot4^N$).

Based on the result for $\rho_{r,3}(K_n)$ in the literature, one might have hoped that $\rho_{r,4}(K_n)$ would always be close to this trivial lower bound, since this is the most natural generalization. 

Our result shows that this is not the case for $r$ at least $25$. In fact, there is another lower bound for $\rho_{r,4}(K_n)$ that is produced from a different (bipartite) ``extremal configuration''.

Fix a partition $V(K_n) = A\cup B$ where $|A| = \left\lfloor \frac{n}{2} \right\rfloor$ and $|B| = \left\lceil \frac{n}{2} \right\rceil$, fix a color $i\in [r]$ and consider the colorings of $K_n$ such that all edges inside $A$ and all edges inside $B$ receive color $i$, with no restriction on the other edges. It is clear that any copy of $K_4$ in such a coloring contains at least two edges of color $i$ and is therefore not rainbow. Therefore,
\[	
	\rho_{r,4}(K_n) \ge r^{\lfloor n^2/4 \rfloor +1}.
\]

To improve this lower bound, it would be desirable to sum the colorings over all possible partitions $\{A,B\}$ (not necessarily balanced) of $V(K_n)$. But, as before, if we did this, we would count some colorings multiple times and, now, the trick with the surjective maps does not work, since we are allowing all $r$ colors to be used on edges between $A$ and $B$. Instead, for each unordered bipartition $\{A,B\}$, fix a spanning tree, $T_{A,B}$, of the complete bipartite graph $K(A,B)$ (induced by the edges from $A$ to $B$). After selecting $\{A,B\}$, consider only the colorings where all edges inside $A$ and $B$ receive color $1$ (not any $i$), edges in $T_{A,B}$ receive any color in $\{2, 3,\ldots, r\}$ and all remaining edges receive any color in $[r]$. 

For a fixed bipartition \(\{A,B\}\), this gives
\[	
	(r-1)^{n-1}r^{|A||B|-(n-1)} = \left(\frac{r-1}{r}\right)^{\!n-1}r^{|A||B|}
\]
colorings.

We now show that the families of colorings obtained from distinct bipartitions are disjoint. Given one of the colorings above, consider the graph \(F\) formed by the edges whose color is \textbf{not} \(1\). By construction, $F$ is a connected bipartite graph. Since a connected bipartite graph has a unique bipartition, up to exchanging its two parts, the unordered bipartition \(\{A,B\}\) is determined by the coloring. Therefore, different unordered bipartitions give disjoint families of colorings. This shows that
\begin{equation}	
	\label{main:eq2}
	\rho_{r,4}(K_n)
	\ge \frac 12 \left(\frac{r-1}{r}\right)^{\!n-1} \sum_{k=1}^{n-1}\binom nk r^{k(n-k)}
	\ge
	\frac12\binom{n}{\lfloor n/2 \rfloor}
	\left(1 - \frac{1}{r}\right)^{n-1} r^{\lfloor n^2/4 \rfloor},
\end{equation}
for all  $r\ge 2$. 

For $r\ge 5$, both lower bounds in equations \eqref{main:eq1} and \eqref{main:eq2} work. For every fixed $r\ge25$, as $r^{n^2/4} > 5^{\binom{n}{2}}$, the lower bound in~\eqref{main:eq2} exceeds the number of colorings using at most five colors by a factor tending to infinity as $n\to\infty$. Thus the five-color asymptotic formula does not extend to this range. This proves the lower bounds in Theorem~\ref{thm:main}.

Using $\binom{n}{\lfloor n/2 \rfloor} = \Theta\left( \frac{2^n}{\sqrt{n}} \right)$ and dividing~\eqref{main:eq2} by \(r^{n^2/4}\), we obtain
\[	
	\frac{\rho_{r,4}(K_n)}{r^{n^2/4}}
	\ge \frac{c}{\sqrt{n}} \left(\frac{2(r-1)}{r} \right)^{\!n},
\]
for some constant $c$. Since $2(r-1)/r > 1$ for $r > 2$, we cannot have an upper bound of the form $(1+\eps)r^{n^2/4}$ in Theorem~\ref{thm:main}\ref{item:main3}.

\begin{remark}	
	For a more general construction, let $K_4^-$ be the graph obtained from a $K_4$ by removing a single edge. Let $G_n$ be any $K_4^-$-free graph on $n$ vertices and $\overline{G}_n$ its complement. This ensures that any four vertices of $G_n$ contain at least two edges of $\overline{G}_n$. Therefore, any $r$-coloring of $K_n$ such that all edges in $\overline{G}_n$ receive the same color is rainbow-$K_4$-free. This yields at least $r\cdot r^{\binom{n}{2}-|E(\overline{G}_n)|}$ colorings. For a graph $H$, define $\ex(n, H)$ as the maximum number of edges in an $H$-free graph on $n$ vertices. The above bound is maximized when $|E(G_n)| = \ex(n, K_4^-)$. It is known (see Corollary~\ref{cor:k4minus-ex}) that, for $n$ large,
	\[	
		\ex(n, K_4^-) = \ex(n, K_3) = \lfloor n^2/4 \rfloor.
	\]
	Thus, if we fix $G_n$ as above, this general construction does not improve the exponent $\lfloor n^2/4 \rfloor +1$ in our lower bound. But this shows that the lower bound in Theorem~\ref{thm:main}~(\ref{item:main3}) can be further improved if we consider all possible choices of $G_n$, provided overlaps between the resulting families are controlled.
\end{remark}

\section{Preliminary results and notation}

\subsection{Extremal Graph tools}
For a graph $H$, let $\chi(H)$ be the chromatic number of $H$. And define $\ex(n, H)$ as the maximum number of edges in a $H$-free graph on $n$ vertices. Let $t_{n,k}$ be the number of edges in the Turán graph with $n$ vertices and $k$ parts. Turán's theorem (1941) states that
\[	
	\ex(n, K_k) = t_{n,k-1},
\]
where $t_{n,k-1}$ is the number of edges on a complete partite graph with $k-1$ parts of sizes differing by at most one. This number is asymptotically  $(1-1/(k-1))\binom{n}{2} + O(n)$.

We say that a graph $H$ is \emph{critical} if it has a \emph{critical edge}, that is, an edge $e$ such that $\chi(H-e)=\chi(H)-1$. We will use the following theorem of Simonovits (Theorem 2.3 in \cite{Simonovits1974extremal}) about critical graphs (that give us a better bound than using Erd\H os-Stone's theorem).

\begin{theorem}[Simonovits - Theorem 2.3 in \cite{Simonovits1974extremal}]
	\label{thm:Simonovits}
	Let \(H\) be a critical graph. Then there exists \(n_0 = n_0(H)\) such that for all \(n \ge n_0\), the Turán graph \(T_{\chi(H)-1}(n)\) is the unique extremal \(H\)-free graph on \(n\) vertices.
\end{theorem}

Let $K_k^-$ be the graph obtained from $K_k$ by removing one edge. As $\chi(K_k^-) = k-1$ and $K_k^-$ is critical, it follows that:

\begin{corollary}\label{cor:k4minus-ex} Let $k\ge 4$ be an integer.  For $n$ large, we have
	\[	
		\ex(n, K_k^-) = \ex(n, K_{k-1}) = t_{k-2}(n).
	\]
	In particular, for $n$ large,
	\[	
		\ex(n, K_4^-) = \ex(n, K_3) = t_{n,2} = \lfloor n^2/4 \rfloor.
	\]
\end{corollary}

We will also use the following well-known result, originally stated by Ruzsa and Szemerédi when $H = K_3$ (Triangle Removal Lemma) as an early consequence of Szemerédi's Regularity Lemma. For a survey, see~\cite{conlon2013graph} and for a proof that avoids using the Regularity Lemma, see~\cite{fox2011new}. And we will use Szemerédi's Regularity Lemma as well.

\begin{lemma}[Graph Removal Lemma]\label{lemma:RL}
	For any fixed graph $H = (V(H), E(H))$ with $v = |V(H)|$ vertices and any $\mu > 0$, there exists a $\delta = \dRL(\mu, H) > 0$ such that the following holds. If $G$ is a graph on $n$ vertices that contains at most $\delta n^v$ copies of $H$, then $G$ can be made $H$-free by removing at most $\mu n^2$ edges.
\end{lemma}

\begin{lemma}[Szemerédi's Regularity Lemma]\label{lem:regularidade}
	For every $\eps > 0$ and integer $m_0 \ge 1$, there exist integers $M = M(\varepsilon, m_0)$ and $n_0 = n_0(\eps, m_0)$ such that every graph $G = (V, E)$ on $n \ge n_0$ vertices admits an equitable $\eps$-regular partition. That is, a partition of the vertex set into $m$ classes,
	\[	
		V = V_0 \cup V_1 \cup V_2 \cup \dots \cup V_m,
	\]
	where $m_0 \le m \le M$, such that:
	\begin{enumerate}	
		\item $V_0$ is an exceptional set with $|V_0| \le \varepsilon n$,
		\item all sets $V_i$ have the same size, i.e., $|V_1| = |V_2| = \dots = |V_m|$,
		\item all but at most $\varepsilon \binom{m}{2}$ pairs $(V_i, V_j)$ with $1 \le i < j \le m$ are $\varepsilon$-regular.
	\end{enumerate}
\end{lemma}

% \begin{theorem}[Removal lemma for \(K_4^-\)]
% \label{thm:k4minus-removal}
% For every \(\mu>0\), there exists
% \[
% \delta_{\mathrm{rem}}(\mu)>0
% \]
% such that if a graph \(H\) on \(n\) vertices has at most
% \[
% \delta_{\mathrm{rem}}(\mu)n^4
% \]
% copies of \(K_4^-\), then one can delete at most \(\mu n^2\) edges from \(H\) to
% make it \(K_4^-\)-free.
% \end{theorem}

% \begin{theorem}[Triangle removal lemma]
% \label{thm:triangle-removal}
% For every \(\mu>0\), there exists
% \[
% \delta_{\triangle}(\mu)>0
% \]
% such that if a graph \(H\) on \(n\) vertices has at most
% \[
% \delta_{\triangle}(\mu)n^3
% \]
% triangles, then one can delete at most \(\mu n^2\) edges from \(H\) to make it
% triangle-free.
% \end{theorem}

\subsection{Bounding pairs of non-edges with relevant co-degrees}

For a graph $H$ and distinct $x, y \in V(H)$, denote by $d_H(x,y)$ the \emph{codegree} of $x$ and $y$, that is, the number of common neighbors of $x$ and $y$ in $H$.

Overview of this section: Let $G$ be a triangle-free graph. Turán's theorem for $K_3$ (also known as Mantel's theorem) implies that the number of non-edges in $G$ cannot be much smaller than the number of edges: $e(\overline{G}) \ge e(G) - O(n)$. And this is tight when $G$ is dense. In Lemma~\ref{lem:triang-codegrees} below, we will show that many pairs of vertices which are not edges of $G$ have a linear co-degree. And in Lemma~\ref{lem:k4minus-codegrees} we arrive at a similar result assuming only that $G$ has few copies of $K_4^-$. Lemma~\ref{lem:k4minus-codegrees} follows from Lemma~\ref{lem:triang-codegrees} together with the Graph Removal Lemma (Lemma~\ref{lemma:RL}). While Lemma~\ref{lem:triang-codegrees} follows from applying the next lemma to a reduced graph obtained from Szemeredi's Regularity Lemma.

\begin{lemma}	
	\label{lem:finite-codegress}
	Let \(R\) be a triangle-free graph. Let \(s_2(R)\) be the number of unordered
	non-adjacent pairs of vertices with at least one common neighbor, and let \(\overline{i}(R)\) be the number of non-isolated vertices of $R$. Then
	\[	
		s_2(R)+\frac{\overline{i}(R)}2\ge e(R).
	\]
\end{lemma}

\begin{proof}	
	Let \(M\) be a maximal matching of \(R\), and write \(|M|=q\). Since the \(2q\)
	vertices covered by \(M\) are non-isolated,
	\[	
		2q\le \overline{i}(R).
	\]
	
	We inject \(E(R)\setminus M\) into the set of unordered non-adjacent pairs of vertices with a common neighbor.

	Let $e = xy \in E(R)\setminus M$. Since $M$ is maximal, at least one of $x, y$ belongs to $V(M)$. Select an arbitrary ordering on the edges of $M$. Assume that \(x\) is not covered by \(M\) and $y$ is covered by the edge $yy' \in M$. Assign \(xy\) to the pair \(\{x,y'\}\). This pair is not an edge, for otherwise \(xyy'\) would be a triangle, and it has common neighbor \(y\). Now, suppose that both $x$ and $y$ are covered by $M$, say, $xx', yy'\in E(M)$ are the unique edges of $M$ that cover them. Assume, without loss o generality, that the edge $xx'$ comes before $yy'$ in the ordering of $M$. Assign \(xy\) to the pair \(\{y,x'\}\). Since \(R\) is triangle-free, $\{y,x'\}$ is a non-edges. Also, this pair has a common neighbor, namely $x$. Note that the order on $M$ ensure that this map is injective, since the graph induced by $x, y, x', y'$ has at most 4 edges and $x'y'$ (in case it is an edge) can be mapped to $\{y', x\}$.
	
	The assignment is injective. Hence
	\[	
		s_2(R)\ge e(R)-q.
	\]
	Using \(q\le \overline{i}(R)/2\) the statement in the lemma follow.
\end{proof}

Just for fun, we show that Lemma~\ref{lem:finite-codegress} gives us yet another\footnote{The authors are not aware if this exact proof or if Lemma~\ref{lem:finite-codegress} appears in other published work, though there are dozens of different proofs of Mantel's theorem so it is hard to search for all of them.} proof for Mantel's theorem.

\begin{corollary}[Mantel's theorem]
	If $R$ is a triangle-free graph on $n$ vertices then $e(R) \le \lfloor n^2/4 \rfloor$. 
\end{corollary}
\begin{proof}	
	Fix $n$. If $R$ is a triangle-free graph with the maximum number of edges. Of course, $\overline{i}(R) \le n$ and $s_2(R) \le e(\overline{R})$. Lemma~\ref{lem:finite-codegress} gives:
	\[	
		e(R) \le s_2(R) + \frac{\overline{i}(R)}{2} \le e(\overline{R}) + \frac n2.
	\]
	Since $e(\overline{R}) = \binom n2 - e(R)$, it follows that
	\[	
		e(R) \le \frac{1}{2}\left(\binom n2 + \frac n2\right) = \frac{n^2}{4}.\qedhere
	\]
\end{proof}

The next lemma is useful only if $e(G) \ge \eta n^2$, but it is true regardless of that.

\begin{lemma}[Co-degrees in triangle-free graphs]
	\label{lem:triang-codegrees}
	For every \(\eta>0\), there exist
	\[	
		\tau = \tau_{\mathrm{tco}}(\eta)>0,
		\quad
		n_0 = n_{\mathrm{tco}}(\eta),
	\]
	such that if \(G\) is triangle-free on \(n\ge n_0(\eta)\)
	vertices then
	\[	
		\left|\{xy\notin E(G):d_G(x,y)\ge \tau n\}\right|
		\ge e(G)-\eta n^2.
	\]
\end{lemma}
\begin{proof}	
	We can assume $\eta < 1$. Choose parameters
	\[	
		0<\eps \ll \delta = \eta/8,
	\]
	For example $0<\varepsilon\le\min\{\eta/100,\delta^2/100\}$. 
	
	The Szemerédi's Regularity Lemma (Lemma~\ref{lem:regularidade}) for such $\eps$ and $m_0=\lceil 1/\eps \rceil $, returns parameters $M = M(\eps)$, $n_0 = n_0(\eps)$ such that every graph $G$ on at least $n_0$ vertices has an $\eps$-regular equitable partition
	\[	
		V(G)=V_0\cup V_1\cup\cdots\cup V_m
	\]
	where $m_0 \le m \le M$. Let $t = |V_i|$ for every $i\neq 0$. We have that
	\[	
		t =\frac{n-|V_0|}{m} \ge \frac{(1-\eps)n}{M}.
	\]
	
	Define a reduced graph \(R\), where \(ij\in E(R)\) if the pair \((V_i,V_j)\) is
	\(\eps \)-regular with density at least \(\delta\).
	
	The constants are chosen so that the number of edges lying inside parts, in between irregular pairs, with an endpoint in $V_0$ or in between regular pairs of density below \(\delta\),
	is at most
	\[	
		3\eps n^2 + \delta n^2 \le \frac {\eta}{2}n^2
	\]
	Thus
	\[	
		e(R)\,t^2
		\ge
		e(G)-\frac{\eta}{2}n^2.
	\]
	Since \(G\) is triangle-free, by the standard embedding lemma, we have that \(R\) is triangle-free, as \(\eps\) is small enough compared to \(\delta\).
	
	By Lemma~\ref{lem:finite-codegress},
	\[	
		s_2(R)+\frac{m}2 \ge s_2(R)+\frac{\overline{i}(R)}2\ge e(R).
	\]
	
	If \(ij\) is a non-edge of \(R\) with a common neighbor \(k\), then the pairs
	\((V_i,V_k)\) and \((V_j,V_k)\) are regular of density at least \(\delta\). All, but at most $\eps t$ vertices of $V_i$ have at least $(\delta -\eps)t \ge \eps t$ neighbors in $V_k$. Let $u$ a typical vertex in $V_i$ (that has such a large neighborhood), $U = (N_G(u)\cap V_k) \subset V_k$. Also, all but at most $\eps t$ vertices of $V_j$ have at least $(\delta -\eps)|U|$ neighbors in $U$, otherwise their set, together with $U$, would violate regularity of $(V_j, V_k)$. Consequently at least
	\[	
		(1-\varepsilon)^2t^2\ge(1-2\varepsilon)t^2
	\]
	pairs in $V_i \times V_j$ have codegree at least
	\[	
		(\delta-\varepsilon)^2t
		\ge\frac{\delta^2}{4}t
		\ge\frac{\delta^2(1-\varepsilon)}{4M}n = \tau n,
	\]
	by choosing $\tau = \frac{\delta^2(1-\varepsilon)}{4M}$.
	% \[
	% |N_G(x) \cap N_G(y)| \ge \frac{\delta^2}{8}|V_k| \ge \frac{\delta^2(1-\eps)}{8M}n.
	% \]
	All these pairs are non-edges of $G$, because $G$ is triangle-free. And the total number of such pairs is at least
	\begin{align*}	
		(1-2\varepsilon)\,s_2(R)\,t^2 & \ge (1-2\varepsilon)\left( e(R) - \frac{m}{2} \right)t^2                 \\
		                              & \ge e(R)t^2-\frac{mt^2}{2}-2\varepsilon e(R)t^2                          \\
		                              & \ge \left(e(G)-\frac\eta2n^2\right)-\frac\varepsilon2n^2-\varepsilon n^2 \\
		                              & \ge e(G)-\eta n^2.\qedhere
	\end{align*}
\end{proof}

\begin{definition}
	Given graphs $H$ and $G$, we denote by $\#(H \subseteq G)$ the number of copies of $H$ in $G$, that is, the number of unlabeled subgraphs of $G$ that are isomorphic to $H$.
\end{definition}

\begin{lemma}[Co-degrees in \(K_4^-\)-sparse graphs]
	\label{lem:k4minus-codegrees}
	For every \(\eta>0\), there exist
	\[	
		\tau_{\mathrm{cod}}(\eta)>0,
		\quad
		\delta_{\mathrm{cod}}(\eta)>0,
		\quad \text{ and } \quad
		n_{\mathrm{cod}}(\eta)
	\]
	such that if \(G\) is a graph on \(n\ge n_{\mathrm{cod}}(\eta)\) vertices
	with
	\[	
		(\#[K_4^- \subseteq G]) \le \delta_{\mathrm{cod}}(\eta)n^4,
	\]
	then
	\[	
		\left|\{xy\notin E(G):d_G(x,y)\ge \tau_{\mathrm{cod}}(\eta)n\}\right|
		\ge e(G)-\eta n^2.
	\]
\end{lemma}

\begin{proof}	
	Let $\eta_1 = \eta/5$ and
	\[	
		\delta_{\mathrm{cod}}(\eta)
		\le
		\dRL(\eta_1, K_4^-).
	\]
	Take
	\[	
		n \ge n_{\mathrm{cod}}(\eta) = \max\left\{n_{\rm tco}(\eta_1),\,
		\frac1{\delta_{\rm rem}(\eta_1,K_3)}\right\}.
	\]
	And let $G$ be a graph as in the statement. By Lemma~\ref{lemma:RL}, deleting at most \(\eta_1 n^2\) edges from \(G\), we obtain a \(K_4^-\)-free graph \(G_0\).
	
	In a \(K_4^-\)-free graph, every edge has codegree at most one. Hence \(G_0\) has at most \(n^2\) triangles. Since \[	
		n^2 \le \dRL(\eta_1, K_3)n^3,
	\]
	by Lemma~\ref{lemma:RL} applied to $K_3$, deleting at most another \(\eta_1 n^2\)
	edges, we obtain a triangle-free graph \(G_1\subseteq G_0\). Then
	\[	
		e(G_1)\ge e(G)-2\eta_1n^2.
	\]
	Apply Lemma~\ref{lem:triang-codegrees} to \(G_1\) with parameter
	\(\eta_1\). Let
	\[	
		\tau_{\mathrm{cod}}(\eta)=\tau_{\mathrm{tco}}(\eta_1).
	\]
	We obtain at least
	\[	
		e(G_1)-\eta_1n^2
	\]
	non-edges of \(G_1\) with codegree at least
	\[	
		\tau_{\mathrm{cod}}(\eta)n
	\]
	in \(G_1\), and hence also in \(G\).
	
	At most \(2\eta_1n^2\) of these pairs are edges of \(G\), since
	\(G\setminus G_1\) has at most \(2\eta_1n^2\) edges. Therefore the number of
	non-edges of \(G\) with the required codegree is at least
	\[	
		e(G)-2\eta_1n^2-\eta_1n^2-2\eta_1n^2
		=
		e(G)-\eta n^2.\qedhere
	\]
\end{proof}

\subsection{Colorings with many monochromatic \texorpdfstring{$K_4$'s}{K4s}}

In the next simple lemma, we consider a $q$-coloring\footnote{We are using the letter $q$ instead of $r$ in the statement of this lemma because later we are going to apply it for some coloring where $q = \binom{r}{5}+1$.} of $K_n$ and we show that \emph{if most copies of $K_4$'s are monochromatic, then there is a color that is used in most edges of~$K_n$}.

\begin{lemma}	
	\label{lem:almost-mono-k4}
	Let \(q\ge2\) and \(0<\alpha<1\). Put
	\[	
		\delta_{\mathrm{mono}}(\alpha,q)=\frac{\alpha}{200q} \text{ and } n_{\mathrm{mono}}(q) = 8q.
	\]
	If $n > n_{\mathrm{mono}}$ and the edges of \(K_n\)
	are colored with \(q\) colors and at most
	\[	
		\delta_{\mathrm{mono}}(\alpha,q)n^4
	\]
	copies of \(K_4\) are non-monochromatic, then some color appears on at least
	\[	
		(1-\alpha)\binom n2 \text{ edges. }
	\]
\end{lemma}

\begin{proof}	
	Let \(N=\binom n2\). Suppose no color appears on at least \((1-\alpha)N\) edges.
	
	Let \(D\) be a largest color class (which we will prove is `dominant'). Since there are \(q\) colors,
	\[	
		|D|\ge \frac{N}{q}.
	\]
	Let
	\[	
		R=E(K_n)\setminus D.
	\]
	By assumption,
	\[	
		|R|\ge \alpha N.
	\]
	
	Count ordered pairs \((e,f)\) of edges such that
	\[	
		e\in R,\quad f\in D,
	\]
	and \(e\) and \(f\) are vertex disjoint. For each fixed \(e\in R\), at most \(2(n-2)\)
	edges touch~\(e\). Hence the number \(P\) of such ordered disjoint pairs
	satisfies
	\[	
		P\ge |R|\bigl(|D|-2(n-2)\bigr).
	\]
	Since \(n > 8q \), we have
	\[	
		|D|-2(n-2)\ge \frac{N}{2q}.
	\]
	Therefore
	\[	
		P\ge \frac{\alpha}{2q}N^2.
	\]
	For \(n\ge3\), \(N\ge n^2/3\), so
	\[	
		P\ge \frac{\alpha}{18q}n^4.
	\]
	
	Every ordered disjoint pair \((e,f)\), with \(e\in R\) and \(f\in D\), spans a
	non-monochromatic copy of \(K_4\). Conversely, a fixed \(K_4\) contains at most
	six ordered disjoint pairs of this type. Thus the number of non-monochromatic
	copies of \(K_4\) is at least
	\[	
		\frac{P}{6}
		\ge
		\frac{\alpha}{108q}n^4
		>
		\frac{\alpha}{200q}n^4,
	\]
	contradicting the hypothesis.
\end{proof}

\subsection{Containers Theorem for color templates}

The Hypergraph Containers Method was developed (independently) in \cite{balogh2015independent} and \cite{saxton2015hypergraph}, for counting independent sets in certain hypergraphs. Since many problems in extremal combinatorics can be reduced to counting independent sets in a appropriately defined hypergraph, the method has been applied to a very large number of problems, in particular, it was used by Balogh and Li \cite{balogh2019typical} to compute $\rho_{r,3}(n)$ up to a small error. The hypergraph that they defined, can be easily generalized from $K_3$-rainbow-free to $K_k$-rainbow-free colorings. When $k=4$, this is the \(6\)-uniform hypergraph whose vertices pairs \((e,c)\), where $e \in E(K_n)$ and $c\in [r]$ represents a colored edge, whose hyperedges are rainbow copies of \(K_4\) in $K_n$. Similar ideas date back to \cite{falgas2019multicolor}.

The exact version of the container theorem that we need, has already been proved and can be found in~\cite{han2025maximumK4} (Theorem 2.5). To state it, we need a few definitions.

\begin{definition}	
	An \(r\)-template of order \(n\) is a function
	\[	
		T:E(K_n)\to 2^{[r]}.
	\]
	A coloring \(\varphi:E(K_n)\to [r]\) is generated by \(T\) if
	\[	
		\varphi(e)\in T(e)
	\]
	for every edge \(e\in E(K_n)\). A good template is one that generates at least one coloring, i.e., $T(e)$ is not empty for every $e$.
\end{definition}

The set $T(e)$ is also called the \emph{palette} of $e$. We say that $T'$ is a subtemplate of $T$, and write $T' \subseteq T$, if $T'(e) \subseteq T(e)$ for each $e \in E(K_n)$. 

Let $\N(T) = \prod_{e\in E(K_n)} |T(e)|$ be the number of colorings that are generated by $T$. And let \(\NR(T)\) be the number of the generated colorings with no rainbow \(K_4\). Let \(\Rfour(T)\) be the number of rainbow-$K_4$ copies in \(T\), that is,
the number of choices of four vertices and six pairwise distinct colors, one from each of the six edge palettes induced by these vertices.

\begin{theorem}[Theorem 2.5 in \cite{han2025maximumK4}]\label{thm:contairs2025Hoppen}
	For every $r \ge 6$, there exist $n_0 = n_0(r)$ and a constant $c(r)$ such that the following hold. For any graph $G$ on $n$ vertices, $n \ge n_0$, there is a collection $\mathcal{C}$ of $r$-templates of $G$ such that:
	\begin{enumerate}	
		\item Every rainbow-$K_4$-free $r$-coloring of $G$ is generated by some $T \in \mathcal{C}$.
		\item For every $T \in \mathcal{C}$, the number \(\Rfour(T)\) of rainbow-$K_4$ copies in $T$ is at most $r^4 n^{-1/3} \binom{n}{4}$.
		\item $|\mathcal{C}| \le 2^{cn^{-1/3} \binom{n}{2} \log^2 n}$.
	\end{enumerate}
\end{theorem}

As a matter of fact, we will only need the following simpler statement. Note that we can delete from the family \(\mathcal C\) all templates that are not good.

\begin{theorem}[Container theorem for templates]
	\label{thm:containers}
	Fix \(r\ge6\). For every \(\xi>0\), there exists \(n_{\mathrm{con}}=n_{\mathrm{con}}(r,\xi)\)
	such that, for every \(n\ge n_{\mathrm{con}}\), there is a family
	\(\mathcal C\) of \emph{good} \(r\)-templates of \(K_n\) such that:
	\begin{enumerate}[label=\textup{(\roman*)}]
		\item every \(r\)-coloring of \(E(K_n)\) with no rainbow \(K_4\) is generated by
		      some \(T\in\mathcal C\);
		\item every \(T\in\mathcal C\) satisfies $\Rfour(T)\le \xi n^4;$
		\item $|\mathcal C|\le \exp(\xi n^2).$
	\end{enumerate}
\end{theorem}

Fixing a template $T$, for $i \in [r]$, let $E^T_i \subset E(K_n)$ be the set of all edges of $K_n$ with exactly $i$ colors in their palettes. Moreover, define $G_i^T = (V(K_n), E_i^T)$ and $G_{i^+}^T = (V(K_n), E_i^T \cup \cdots E_r^T)$, the graph formed by edges with at least $i$ colors. When \(T\) is clear, we write simply \(G_{i+}\). %Let $v \in V(K_n)$, we denote $N_i^T(v)$ as the neighborhood of $v$ in the graph $G_i^T$. We analogously define $N_{i^\pm}^T(v)$. Finally, when the template $T$ is clear from the context, we will omit $T$ from the notation.
% For a template \(T\), define
% \[
% E(G^T_{6+})=\{e\in E(K_n): |T(e)|\ge 6\}.
% \]

%%% We don't use the following definition explicitly anymore. 
% \begin{definition}
% We say that an $r$-template $T$ is $g(n)$-\emph{feasible} if
% \begin{enumerate}
% 	\item For every $e \in E(K_n)$, we have $|T(e)| \ge 1$; and
% 	\item $R_4(T) \le r^4\,g(n)\binom{n}{4}$.
% \end{enumerate}
% When $g(n)$ is clear, we will refer to it as a feasible template instead of a $g(n)$-feasible template.
% \end{definition}

\subsection{Lemmas related to templates}

\begin{proposition}[Five high edges]
	\label{lem:k4minus-template}
	Let \(T\) be a good \(r\)-template and let
	\[	
		H=G^T_{6+}.
	\]
	Then every $K_4^-$ generates at least one rainbow $K_4$ inside $T$. More precisely,
	\[	
		6\Rfour(T) \ge (\#[K_4^- \subseteq H]),
	\]
\end{proposition}

\begin{proof}	
	A copy of \(K_4^-\) in \(H\) has five edges with palettes of size at least six and a pair of vertices which may or may not be in $H$. Since every palette is non-empty, the possibly missing edge still has at least one color available. Choose the color of that edge first. On the remaining edges, greedily choose a color different from those that were previously chosen. Thus the four vertices support a rainbow \(K_4\)-subtemplate. A fixed rainbow \(K_4\) can be counted for at most six choices of the possibly missing edge (in case it is not actually missing).
\end{proof}

\begin{lemma}[Singleton compensation]
	\label{lem:singleton-compensation}
	Fix \(r\ge6\) and \(\eta>0\). There exist
	\[	
		\delta_{\mathrm{sing}}(r,\eta)>0,
		\quad
		n_{\mathrm{sing}}(r,\eta)
	\]
	such that the following holds. Let \(T\) be a good \(r\)-template such that
	\[	
		\Rfour(T)\le \delta_{\mathrm{sing}}(r,\eta)n^4,
	\]
	then at least \(e(G^T_{6+})-\eta n^2\) non-edges of \(G^T_{6+}\) are singleton edges of \(T\).
\end{lemma}

\begin{proof}	
	Let $T$ be a good template, $H = G^T_{6+}$ and $m = e(H)$. Choose
	\[	
		\delta_{\mathrm{sing}}(r,\eta)
		= \min \left\{\frac16\,\delta_{\mathrm{cod}}(\eta/2),\;  \frac{\eta\,(\tau_{\mathrm{cod}}(\eta/2))^2}{48}\right\}
		.
	\]
	By Proposition~\ref{lem:k4minus-template},
	\[	
		(\#[K_4^- \subseteq H])\le6\Rfour(T)
		\le
		6\delta_{\mathrm{sing}}(r,\eta)n^4
		\le
		\delta_{\mathrm{cod}}(\eta/2)n^4.
	\]
	Lemma~\ref{lem:k4minus-codegrees} (codegrees in $K_4$-sparse graphs) gives a set \(F\) of non-edges of \(H\) such that
	\[	
		|F|\ge m-\frac{\eta}{2}n^2
	\]
	and every \(xy\in F\) satisfies
	\[	
		d_H(x,y)\ge \tau_{\mathrm{cod}}(\eta/2)n.
	\]
	
	Let \(F'\subseteq F\) be the set of pairs \(xy\) with \(|T(xy)|\ge2\). For each
	\(xy\in F'\), and for every pair
	\[	
		\{u,v\}\subseteq N_H(x)\cap N_H(y),
	\]
	the four edges
	\[	
		xu,\ xv,\ yu,\ yv
	\]
	have palettes of size at least six. The edge \(xy\) has at least two colors, and
	the edge \(uv\) has at least one. Choose distinct colors on \(xy\) and \(uv\),
	and then complete the four high edges greedily. This gives a rainbow
	\(K_4\)-subtemplate.
	
	For \(n\) sufficiently large, each \(xy\in F'\) gives at least
	\[	
		\frac{(\tau_{\mathrm{cod}}(\eta/2))^2}{4}n^2
	\]
	choices of \(\{u,v\}\). Summing over all possible choices of $xy$, each rainbow \(K_4\)-subtemplate is counted at most six times in this way. Hence
	\[	
		\Rfour(T)
		\ge
		\frac{\tau_{\mathrm{cod}}(\eta/2)^2}{24}|F'|n^2.
	\]
	Since $\Rfour(T)\le \delta_{\mathrm{sing}}(r,\eta)n^4$ and by the choice of $\delta_{\mathrm{sing}}(r,\eta)$, 
	we have \(|F'|\le(\eta/2)n^2\). Therefore at least
	\[	
		m-\eta n^2
	\]
	pairs in \(F\) are singleton edges of \(T\).
\end{proof}

\section{The bipartite regime (\texorpdfstring{$r\ge 25$}{r>=25})}

The following lemma proves the upper bound in Theorem~\ref{thm:main}(\ref{item:main3}).

\begin{lemma}	
	\label{lemma:above}
	Fix an integer \(r\ge 25\). For every \(\eps>0\), there exists
	\(n_0(r,\eps)\) such that, for every \(n\ge n_0(r,\eps)\),
	\[	
		\rho_{r,4}(K_n)\le r^{(1/4+\eps)n^2}.
	\]
\end{lemma}

\begin{proof}	
	Fix \(r \ge 25\) integer and \(\eps>0\). Let
	\[	
		\ell=\log_r5 \le \frac 12.
	\]
	Choose
	\[	
		0<\eta =
		\frac{\eps}{2\bigl(1-\ell\bigr)}.
	\]
	Choose \(\xi>0\) satisfying all of the following:
	\[	
		\xi\le \delta_{\mathrm{sing}}(r,\eta),
		\quad 
		6\xi\le \dRL(\eta, K_4^-),
		\quad \text{and} \quad
		\exp(\xi n^2)\le r^{\eps n^2/2}
	\]
	for sufficiently large \(n\). Let $n > n_{\mathrm{con}}(r, \xi)$ and apply Theorem~\ref{thm:containers} (Containers) to $K_n$ with this \(\xi\) to obtain a family $\mathcal{C}$ of templates.
	Since all rainbow-$K_4$-free colorings of $K_n$ are generated by some template in $\mathcal{C}$ and $|\mathcal{C}| \le r^{\eps n^2/2}$, it suffices to show that every container \(T\) in the family $\mathcal{C}$ generates (at most)
	\[	
		\NR(T) \le \N(T) \le r^{(1/4+\eps/2)n^2} \text{ colorings}.
	\]
	
	Fix such a good template \(T\). We know that $\Rfour(T) \le \xi n^4$. Let
	\[	
		H=G^T_{6+} \quad \text{and} \quad m=e(H).
	\]
	
	By Lemma~\ref{lem:singleton-compensation}, at least
	\[	
		m-\eta n^2
	\]
	non-edges of \(H\) are singleton edges of \(T\) (Note that $m-\eta n^2$ could be negative). Hence
	\[	
		\N(T) = \prod_e|T(e)|
		\le
		r^m5^{N-m-(m-\eta n^2)}
		=
		r^{\ell N+(1-2\ell)m+\ell\eta n^2}.
	\]
	
	If $r=25$, then $\ell = 1/2$ and the proof finishes here. So, assume that $r > 25$. Then \(\ell<1/2\), that is, \(1-2\ell>0\). In this case, we will have to bound $m$ using the $K_4^-$-sparsity of \(H\).
	
	By Proposition~\ref{lem:k4minus-template},
	\[	
		(\#[K_4^- \subseteq H])\le 6\Rfour(T)\le 6\xi n^4.
	\]
	Since \(6\xi\le\dRL(\eta, K_4^-)\), deleting at most \(\eta n^2\)
	edges from \(H\) makes it \(K_4^-\)-free. By Corollary~\ref{cor:k4minus-ex}, for
	\(n\) sufficiently large, the resulting graph has at most $	
		\frac{n^2}{4}$ edges. Therefore
	\[	
		m\le \left(\frac{1}{4}+\eta\right)n^2.
	\]
	
	Since \(1-2\ell>0\), the exponent
	\[	
		\ell N+(1-2\ell)m+\ell\eta n^2
	\]
	is increasing in \(m\). Hence
	\begin{align*}	
		\ell N+(1-2\ell)m+\ell\eta n^2
		 & \le
		\ell\binom n2
		+
		(1-2\ell)\left(\frac{1}{4}+\eta\right)n^2
		+
		\ell\eta n^2.                                                                 \\
		 & = \ell\binom n2+(1-2\ell)\frac{n^2}{4}+\bigl((1-2\ell)+\ell\bigr)\eta n^2  \\
		 & = \left(\frac{n^2}{4}-\frac{\ell n}{2}\right) + \bigl(1-\ell\bigr)\eta n^2 \\
		 & \le \left(\frac{1}{4} + \frac{\eps}{2}\right)n^2,
	\end{align*}
	where the last inequality follows from the choice of $\eta$.
	
	Thus
	\[	
		\N(T) \le r^{\ell N+(1-2\ell)m+\ell\eta n^2} \le
		r^{(1/4+\eps/2)n^2}.
	\]
	
	Every container contributes at most 
	$r^{(1/4+\eps/2)n^2}$,
	and the family of containers has size at most
	$r^{\eps n^2/2}$. Hence
	\[	
		\rho_{r,4}(K_n)\le r^{(1/4+\eps)n^2}.\qedhere
	\]
\end{proof}

\section{The five-color regime (\texorpdfstring{$6\le r \le 24$)}{6<= r<=24}}

In this section we aim to prove Theorem~\ref{thm:main}(\ref{item:main2}). In this regime, the result is sharper, since the error term, $\eps$, is in the multiplicative constant instead of  the exponent. Unsurprisingly, the proof is more involved. 

\subsection{A structural lemma for templates}

\begin{lemma}[Structural lemma below the threshold]
	\label{lem:structural-below}
	Fix \(6\le r\le24\) and \(0<\alpha<1\). There exist
	\[	
		\zeta_{\mathrm{str}}(r,\alpha)>0,
		\quad
		\delta_{\mathrm{str}}(r,\alpha)>0,
		\quad
		n_{\mathrm{str}}(r,\alpha)
	\]
	such that for every $n\ge n_{\mathrm{str}}(r,\alpha)$ every good \(r\)-template \(T\) of $K_n$ (with non-empty palettes) with
	\[	
		\Rfour(T)\le \delta_{\mathrm{str}}(r,\alpha)n^4
	\]
	satisfies one of the following alternatives:
	\begin{enumerate}[label=\textup{(\roman*)}]
		\item      
		      \[	
			      \prod_{e\in E(K_n)} \!|T(e)|\le 5^{N-\zeta_{\mathrm{str}}(r,\alpha)n^2};
		      \]
		\item there exists \(C\in\binom{[r]}5\) such that
		      \[	
			      |\{e:T(e)=C\}|\ge (1-\alpha)N.
		      \]
	\end{enumerate}
\end{lemma}

\begin{proof}	
	Let $6\le r \le 24$ and $\alpha > 0$ as in the statement. Choose 
	\[	
		q = \binom{r}{5}+1, \qquad \beta = \min\Big\{\frac{\alpha}{200} , \frac{1}{50}\delta_{\mathrm{mono}}(\alpha,q)\Big\}.
	\] 
	Choose 
	\[	
		\zeta = \zeta_{\mathrm{str}}(r,\alpha) = \frac 12(2 - \log_5 r)\beta,
	\]
	and note that $0 < \zeta < \beta$ since $1 < \log_5 r < 2$. We will apply Lemma~\ref{lem:singleton-compensation}, choosing $\eta = \zeta$. So, finally, set
	\[	
		\delta_{\mathrm{str}}(r,\alpha) = \min\Big\{\delta_{\mathrm{sing}}(r,\eta), \frac12\delta_{\mathrm{mono}}(\alpha,q)\Big\}.
	\]
	
	Assume that $n$ is sufficiently large. Take a good template $T$ of $K_n$ satisfying $\Rfour(T)\le \delta_{\mathrm{str}}(r,\alpha)n^4$. Let
	\[	
		H=G^T_{6+},\quad m=e(H).
	\]
	
	\noindent\textbf{Case 1:} First suppose $m\ge \beta n^2$. Since $\delta_{\mathrm{str}}(r,\alpha) \le
		\delta_{\mathrm{sing}}(r ,\eta )$
	we can apply Lemma~\ref{lem:singleton-compensation} (singleton compensation) to~$H$, which gives at least
	\[	
		m-\eta n^2
	\]
	singleton edges in \(T\). Hence
	\[	
		\prod_e|T(e)|
		\le
		r^m5^{N-2m+\eta n^2}
		=
		5^{N-(2 - \log_5r)m+\eta n^2} \le 5^{N-2\zeta n^2 + \eta n^2} = 5^{N-\zeta n^2},
	\]
	since \(m\ge \beta n^2\). 
	Thus the alternative (i) in the theorem statement holds.
	
	\noindent\textbf{Case 2:} Now suppose
	\[	
		m<\beta n^2.
	\]
	Let \(a\) be the number of edges outside \(H\) whose palettes have size at most
	\(4\). Then
	\[	
		\prod_e|T(e)|
		\le
		r^{\beta n^2}5^{N-m - a}4^a.
	\]
	Since $r < 25$ and $(4/5) < 5^{-1/8}$, we have that:
	\[	
		\prod_e|T(e)| \le 5^{N + 2\beta n^2 - (a/8)}.
	\]
	If
	\[	
		a\ge 24\beta n^2,
	\]
	then
	\[	
		\prod_e|T(e)| \le 5^{N-\beta n^2} \le 5^{N-\zeta n^2}
	\]
	and alternative (i) in the theorem statement holds. 
	
	Therefore, we can assume that $a < 24\beta n^2$. So the total number of edges in $T$ whose palette does not have size exactly \(5\) is less than
	\[	
		24\beta n^2+ m \le 25\beta n^2.
	\]
	
	Let \(J = G^T_{5}\) be the graph of edges whose palettes have size exactly \(5\). Color
	each \(e\in E(J)\) by the five-set \(T(e)\in\binom{[r]}5\). Extend this coloring
	to all of \(K_n\) by giving every edge outside \(J\) a new dummy color \(\star\).
	Thus we have an $q$-edge-coloring of \(K_n\) with $q=\binom r5+1$ colors.
	
	Let
	\[	
		B=E(K_n)\setminus E(J).
	\]
	We know $|B|\le 25\beta n^2$. The number of copies of \(K_4\) using at least one edge of \(B\) is at most
	\[	
		|B|\binom{n-2}{2}
		\le
		25\beta n^4 \le \frac 12\delta_{\mathrm{mono}}(\alpha, q) n^4.
	\]
	
	Now consider a \(K_4\) lying entirely in \(J\). If its six 5-palettes are not
	all equal, then those six sets admit a system of distinct representatives.
	Indeed, six sets of size \(5\) can fail Hall's condition only if their union has size at most
	\(5\), and then all six sets must be equal. Hence every non-monochromatic
	\(K_4\subseteq J\) gives a rainbow \(K_4\)-subtemplate. Therefore the number of
	such non-monochromatic \(K_4\)'s is at most
	\[	
		\Rfour(T)\le \delta_{\mathrm{str}}(r,\alpha)n^4.
	\]
	
	Finally, the total number of non-monochromatic \(K_4\)'s in the extended coloring is at most
	\[	
		25\beta n^4 + \delta_{\mathrm{str}}(r,\alpha)n^4 \le \delta_{\mathrm{mono}}(\alpha,q)n^4.
	\]
	
	Then Lemma~\ref{lem:almost-mono-k4} implies that some color in the extended
	coloring appears on at least
	\[	
		(1-\alpha)N \text{ edges.}
	\]
	The color that appear the most (in the extended
	coloring) cannot be the dummy color \(\star\), since \(\star\) appears only
	on \(|B|\le 25\beta n^2 \) edges. Therefore the dominant color is a five-palette
	\[	
		C\in\binom{[r]}5,
	\]
	and
	\[	
		|\{e:T(e)=C\}|\ge (1-\alpha)N. \qedhere
	\]
\end{proof}

\subsection{Counting colorings close to a fixed 5-palette}

\begin{definition}	
	For \(C\in\binom{[r]}5\) and \(\alpha>0\), define
	\[	
		\mathcal F_C(\alpha)
		=
		\left\{
		\varphi:E(K_n)\to[r]:
		\begin{array}{l}
			\varphi\text{ has no rainbow }K_4, \\
			|\{e:\varphi(e)\notin C\}|\le\alpha \binom n2
		\end{array}
		\right\}.
	\]	
\end{definition}

Note that the above definition does not depend on a particular template. It contains the union of all coloring generated by templates for which most palettes are equal to $C$, i.e., those that fall on case (ii) of Theorem~\ref{lem:structural-below} for that particular $C$.

We will use the following simple result which is a particular case of Baranyai’s theorem.

\begin{proposition}	
	\label{prop:baranyai}
	If $t$ is even, the edges of complete graph $K_t$ can be decomposed into $(t-1)$ edge-disjoint perfect matchings. If $t$ is odd, it can be decomposed into $t$ edge-disjoint near-perfect matchings, each of size $(t-1)/2$.
\end{proposition}

\begin{figure}	
	\centering
	% -------------------------------------------------------------
	% PART 1: K_8 Factorization into 7 Perfect Matchings
	% -------------------------------------------------------------
	\begin{tikzpicture}[scale=1.2]
		% \node[font=\Large\bfseries] at (3.5, 2.5) {$K_8$ 1-Factorization (7 Perfect Matchings)};
		\foreach \v in {1,...,7} {	
				\coordinate (V\v) at ({90 - (\v-1)*360/7}:1) {};
				\node[scale=0.7] at ({90 - (\v-1)*360/7}:1.2) {\v};
			}
		\node[circle, draw, fill=black!10, inner sep=0.5pt] (V8) at (0,0) {\tiny 8};
		
		\foreach \c [count=\f from 0] in {vibrant-grey, 	vibrant-blue, 	vibrant-cyan, 	vibrant-green, 	vibrant-orange, 	vibrant-red, 	vibrant-purple} {	
				\begin{scope}[color=\c]
					\pgfmathsetmacro{\centerMatch}{int(\f + 1)}
					\draw[thick] (V8) -- (V\centerMatch);
					
					\foreach \j in {1,2,3} {	
							\pgfmathtruncatemacro{\leftV}{mod(\f + 1 - \j + 7 - 1, 7) + 1}
							\pgfmathsetmacro{\rightV}{mod(\f + 1 + \j - 1, 7) + 1}
							\draw[thick,  line width=(\f+10)/20 pt] (V\leftV) -- (V\rightV);
						}
				\end{scope}
			}
		\foreach \v in {1,...,7} {	
				\node[circle, fill=black!80, inner sep=1pt] at (V\v){};
			}
	\end{tikzpicture}
	\quad
	% -------------------------------------------------------------
	% PART 2: K_7 Factorization into 7 Almost Perfect Matchings
	% -------------------------------------------------------------
	\begin{tikzpicture}[scale=1.2]
		% \node[font=\Large\bfseries] at (3.5, 2.5) {$K_8$ 1-Factorization (7 Perfect Matchings)};
		\foreach \v in {1,...,7} {	
				\coordinate (V\v) at ({90 - (\v-1)*360/7}:1) {};
				\node[scale=0.7] at ({90 - (\v-1)*360/7}:1.2) {\v};
			}
		% \node[circle, draw, fill=black!20, inner sep=0.5pt] (V8) at (0,0) {\tiny 8};
		
		\foreach \c [count=\f from 0] in {vibrant-grey, 	vibrant-blue, 	vibrant-cyan, 	vibrant-green, 	vibrant-orange, 	vibrant-red, 	vibrant-purple} {	
				\begin{scope}[color=\c]
					\foreach \j in {1,2,3} {	
							\pgfmathtruncatemacro{\leftV}{mod(\f + 1 - \j + 7 - 1, 7) + 1}
							\pgfmathsetmacro{\rightV}{mod(\f + 1 + \j - 1, 7) + 1}
							\draw[thick,  line width=(\f+10)/20 pt] (V\leftV) -- (V\rightV);
						}
				\end{scope}
			}
		\foreach \v in {1,...,7} {	
				\node[circle, fill=black!80, inner sep=1pt] at (V\v){};
			} 
	\end{tikzpicture}
	\caption{Matchings factorization of a $K_8$ and $K_7$}
\end{figure}
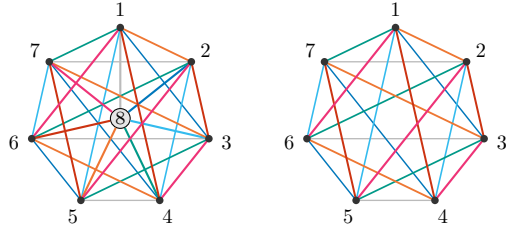

The point of the next lemma is that \(\mathcal F_C(\alpha)\) contains essentially
only the \(5^N\) colorings that use only colors from \(C\) in every edge.

We will also need the natural entropy function, which we will denote by $H(x)$, and is defined as $H(x)=-x\ln(x)-(1-x)\ln(1-x)$, for $0<x<1$. It will be useful for the well known estimate
\begin{eqnarray}\label{eq:entropy}
	\sum_{k=1}^{\lfloor xn \rfloor}\binom{n}{k} \leq e^{H(x)n} \le e^{2x\ln(1/x)n},
\end{eqnarray}
for $x < 1/2$. Note that $\lim_{x \rightarrow 0^+} H(x)=0$.

\begin{lemma}[Fine estimate for $|\mathcal{F}_C(\alpha)|$]
	\label{lem:fine-count}
	Fix \(6\le r\le24\) and \(C\in\binom{[r]}5\). There is an absolute constant $\alpha_{\mathrm{fine}}$ such that for every \(\eps>0\) there exists $n_{\mathrm{fine}}(\alpha_{\mathrm{fine}},\eps)$ such that for all \(n\ge n_{\mathrm{fine}}(\alpha_{\mathrm{fine}},\eps)\)
	\[	
		|\mathcal F_C(\alpha_{\mathrm{fine}})|
		\le
		(1+\eps)5^N, \quad \text{where } N=\binom{n}{2}.
	\]
	Equivalently, the number of colorings in \(\mathcal F_C(\alpha_{\mathrm{fine}})\)
	using at least one color outside \(C\) is at most \(\eps5^N\).
\end{lemma}

\begin{proof}	
	We prove the equivalent statement. Write
	\[	
		S=\{e:\varphi(e)\notin C\}.
	\]
	We count $K_4$-rainbow free colorings for which
	\[	
		1\le |S|\le \alpha N,
	\]
	where \(\alpha\) is chosen below.
	
	A trivial upper bound is given by first choosing the color of the edges of $S$ \[	
		U = \sum_{S} (r-5)^{|S|}5^{N-|S|}  
		\le 5^N\sum_{s=1}^{\alpha N}\binom{N}{s}\left(\frac{r-5}{5}\right)^{\!s}.
	\]
	where the first sum is taken over all non-empty sets \(S\subseteq E(K_n)\) of size at most \(\alpha N\). To improve this, we need to show that after some edges of $S$ have been selected, this drastically restrict the choices for other edges.
	
	We split the possible choices for $S$ into four cases and show that each case contributes at most \((\eps/4)5^N\) colorings. %The first case is when \(t\) is large, the second case is when \(t\) is small and there are few edges between \(A\) and \(B\), and the third case is when \(t\) is small and there are many edges between \(A\) and \(B\).
	
	Choose constants: $\alpha = 10^{-36}$, $\beta = 10^{-12}$, $\gamma = 10^{-6}$ and $K=10^7$. Because $\ln(1/\gamma) < \ln(1/\beta) < \ln(1/\alpha) < 40\ln (10) < 10^2$, those choices imply that:
	\[	
		\gamma\ln(1/\gamma) < 10^{-4}, \quad \beta\ln(1/\beta) < 10^{-4}\gamma, \quad \alpha(\ln(1/\alpha)+3) \le 10^{-9}\beta^2.
	\]
	Furthermore, 
	\[	
		K \alpha \le 10^{-4}\gamma \quad \text{and} \quad (\ln(1/\alpha)+3) \le K/1000.
	\]

	For any set $S$ of edges of $K_n$, let \(M\subseteq S\) be a maximum matching.\footnote{Start fixing any order over all matchings of $K_n$. After selecting $S$, if it contains more than one maximum matching, let $M$ be the first from that order. In that way, the choice of $S$ uniquely determines $M$.} Write
	\[	
		|M|=t,    \quad   A=V(M),  \quad  B=V(K_n)\setminus A.
	\]
	Since \(M\) is maximum,
	\[	
		S\cap E(B)=\emptyset.
	\]
	Thus every edge inside \(B\) must receive a color in \(C\).

	\medskip

	\noindent\textbf{Case 1: \(t\ge K\alpha n\).}
	
	Choose a submatching \(M_0\subseteq M\) of size
	\[	
		t_0=\min\{t,\lfloor n/4\rfloor\}.
	\]
	Let \(B_0\) be the vertices not covered by \(M_0\). Then \(|B_0|\ge \lceil n/2 \rceil\). By Proposition~\ref{prop:baranyai}, the complete graph induced by the vertices in $B_0$ can be factored into at least $|B_0|-1$ edge-disjoint matchings, each with at least $\lfloor n/4 \rfloor$ edges. For each edge \(xy\in M_0\), injectively assign a matching $M(xy)$ of edges in \(B_0\). An edge $xy\in M_0$ together with an edge $uv\in M(xy)$ form a candidate \(K_4\)-block. There are at least $t_0\lfloor n/4 \rfloor$ candidate blocks. Note that the edges that appear in those candidates $K_4$ are distinct, except for those in $M_0$. We say that a candidate block is destroyed if one of its five edges that do not belong to $M_0$ lies in \(S\). Since these five-edge sets are disjoint over the candidate blocks, at most \(|S|\le\alpha N\le \alpha n^2/2\) candidate blocks are destroyed. Hence the number of good blocks is at least
	\[	
		t_0\left\lfloor\frac{n}{4}\right\rfloor-\frac{\alpha n^2}{2}.
	\]
	If \(t<n/4\), then \(t_0=t\ge K\alpha n\), so the number of good blocks is at
	least
	\[	
		\frac{K\alpha n^2}{5},
	\]
	(as $n> 4K$ and $K\ge 15$).
	
	If \(t\ge n/4\), then it is at least \(n^2/20\), for \(\alpha/2 + 4/n < 1/80\). Hence in every case (as $K \alpha < 1/4$) there are at least
	\[	
		\frac{K\alpha n^2}{5}
	\]
	good blocks.
	
	In a good block with external edge \(xy\), the remaining five edges have colors
	in \(C\). They cannot use all five colors of \(C\), otherwise the block would be
	a rainbow \(K_4\). Thus, instead of having $5^5$ ways to color those edges, we have only $5^5 - 5!$ ways. Since the relevant sets of five-edge
	are disjoint, the total loss factor is at most
	\[	
		\left(1-\frac{5!}{5^5}\right)^{K\alpha n^2/5} \le \exp\left(-\frac{5!}{5^5}\frac{K \alpha n^2}{5}\right) \le \exp(-10^{-3}K \alpha n^2).
	\]
	
	The number of choices for \(S\) is less than
	\[	
		\sum_{s=1}^{\alpha N}\binom Ns \le \sum_{s=1}^{\alpha n^2/2}\binom{n^2/2}{s} \le \exp(\alpha n^2\ln(1/\alpha)).
	\]
	The number of ways to color the edges in $S$ is at most 
	\[	
		(r-5)^{|S|}\le 19^{\alpha N} \le \exp(3 \alpha N) \le \exp{(3 \alpha n^2)}.
	\]
	
	Therefore, the total number of colorings in this case is at most
	\begin{align*}	
		\sum_{s=1}^{\alpha N}\binom{N}{s} & (r-5)^{s}\left(1-\frac{5!}{5^5}\right)^{K\alpha n^2/5}5^{N-s} \le \\
		\le                               & \exp\left( (\ln(1/\alpha) + 3 - 10^{-3}K)\alpha n^2 \right)5^N    \\
		\le                               & \exp(-c_1\alpha n^2)5^N 
		\le \frac{\eps}{4}5^N,
	\end{align*}
	for some positive constant $c_1$, by the choice of $\alpha$ and $K$.
	
	\medskip
	
	\noindent\textbf{Case 2: \(t<K\alpha n\) and \(|S\cap E(A,B)|\le\gamma tn\).}
	
	In this case, there are not so many choices for the edges of $S$. Start choosing the matching \(M\), what can be done in less than \(n^{2t}\). The choice of $M$ determines $A$. The number of choices for the colors of all edges inside $A$ (which also determines which edges in $A$ belong to $S$) is at most
	\[	
		r^{\binom{2t}{2}} \le r^{2t^2} \le e^{8t^2}.
	\]
	
	The number of choices for the edges of $S$ between \(A\) and
	\(B\), with their colors, is at most
	\[	
		\sum_{s=0}^{\gamma t n} \binom{2tn}{s}(r-5)^s \le \exp\bigl(4\gamma tn\ln(1/\gamma) + 3\gamma tn\bigr),
	\]
	since $|A||B| \le 2tn$. Altogether, the number of choices so far it at most
	\[	
		\exp\bigl(2t\ln n + 8t^2 + 4\gamma tn\ln(1/\gamma) + 3\gamma tn\bigr).
	\]
	
	All the remaining edges cannot belong to $S$, therefore, their colors have to be chosen from $C$. To finish this case, we have to improve on the  trivial upper bound ($5^{N-s}\le 5^N$) for the number of coloring of the edges not in $S$.
	
	For each edge \(xy\in M\), as in Case 1, injectively choose a matching of edges in \(B\) (by selection a factor in $K_{|B|}$). As $|A| < 2K \alpha n < n/2$, we have $|B| > n/2$. So, this creates at least $t\lfloor \frac{n}{4}\rfloor$ candidates $K_4$'s, whose $5$ edges that do not belong to $M$ are all distinct. Since
	\(|S\cap E(A,B)|\le\gamma tn\), at most \(\gamma tn\) candidate blocks are
	destroyed. Therefore, the number of good blocks is at least
	\[	
		t\left\lfloor\frac{n}{4}\right\rfloor-\gamma tn \ge \frac{tn}{6},
	\]
	as $\gamma < 1/30$ and $n > 20$. Each block gives a fixed loss of $1 - (5!/5^5)$, for a total loss of
	\[	
		\left(1 - \frac{5!}{5^5}\right)^{tn/6} \le \exp\left(-\frac{5!}{5^5\cdot 6} tn\right) \le \exp(-10^{-3} tn).
	\] 

	For $n$ large, as $t \le K \alpha n$, we have that 
	\[	
		\bigl(2t\ln n + 8t^2 + 4\gamma tn\ln(1/\gamma) + 3\gamma tn\bigr) - 10^{-3}tn < -c_2tn,
	\]
	for some $c_2 > 0$. Since $t\ge 1$, for $n$ large enough, the total contribution of Case 2 is attained by summing over all possible values of $t$:
	\[	
		\sum_{t=1}^{\lfloor n/2\rfloor}\exp(-c_2tn)5^N \le \frac{\eps}{4}5^N.
	\]
	
	\medskip
	
	\noindent\textbf{Case 3: \(t<K\alpha n\) and \(|S\cap E(A,B)|>\gamma tn\).}
	
	This case will require a little more work, since the number of good $K_4$'s can be small. We will have to split into two subcases.
	For \(u\in A\), define
	\[	
		X(u)=\{v\in B:uv\in S\},
		\quad
		Y(u)=B\setminus X(u).
	\]
	Call \(u\) large if
	\[	
		|X(u)|\ge\beta n.
	\]
	The non-large vertices contribute at most \(2t\beta n\) edges of $S\cap E(A,B)$. Since \(\beta\ll\gamma\), they contribute less than
	\[	
		\frac{\gamma}{2}tn.
	\]
	Thus large vertices contribute at least \((\gamma/2)tn\) edges of $S\cap E(A,B)$. If \(p\) is the number of large vertices, then
	\[	
		2t \ge p \ge \frac{\gamma}{2}t.
	\]
	
	Put
	\[	
		c_\beta=\frac{\beta}{100r}.
	\]
	
	We say that a coloring that satisfy case 3 is \emph{star-structured} if for every large vertex $u \in A$ there exists \[	
		P_u\subseteq [r]\setminus C,\quad |P_u|\le2,
	\]
	and
	\[	
		a_u\in C
	\]
	such that all but \(\beta n/100\) edges from \(u\) to \(B\) have color in
	\[	
		P_u\cup\{a_u\}.
	\]
	
	\noindent\textbf{Case 3.1.} We will show that the number of colorings which are \textbf{not stars-structured} is at most $(\eps/4)5^N$.
	
	As in Case 1, the number of choices for $S$ and the colors of the edges in $S$ is trivially bounded by
	\[	
		\sum_{s=1}^{\alpha N}\binom{N}{s}(r-5)^s \le \exp\bigl(\alpha n^2\ln(1/\alpha) + 3 \alpha N\bigr).
	\]
	
	From $S$, we have that $M$ is determined, and so are the sets $A$ and $B$. For the coloring not to be star-structured, we have to select at least one $u$ from $A$ to be a large vertex that fails the start-structure condition. The number of ways to select $u$ and color all edges incident to $u$ is less than
	\[	
		n\cdot r^n \le \exp(\ln n + 4n).
	\] 
	
	Let $uB = \{uv : v\in B\}$ be the set of edges from $u$ to $B$. Let $a_u$ be the most used color from $C$ in $uB$. Because $u$ is large, the most used color from $[r]\setminus C$ is usef at least $\frac{\beta n}{r-5}\ge c_\beta n$ times. Let $P_u$ be the set of the two\footnote{In case $r=6$, there is only one color in $[r]\setminus C$; in that case $P_u$ is defined as that single color.} most used colors in $uB$ not from $C$.

	If at most two external colors (from $[r]\setminus C$) and at most one color in $C$ occur more than $c_\beta n$ times in the colored set $uB$, then all colors outside $P_u \cup \{a_u\}$ together occur at most 
	\[	
		r\,c_\beta n =  \frac{\beta n}{100} \text{ times},
	\]
	and $u$ would satisfy the star-structure condition, contradicting our assumption. Therefore, at least one of the following two alternatives must occur:
	\begin{itemize}	
		\item Three colors in $[r]\setminus C$ occur on at least \(c_\beta n\) in \(uB\). Let the corresponding vertex classes be \(B_1, B_2, B_3\).
		      Choose subsets of size
		      \[	
			      q=\lfloor c_\beta n \rfloor
		      \]
		      inside each of \(B_1,B_2,B_3\), and index them as
		      \[	
			      B_1=\{x_i:i\in\mathbb Z_q\},\quad
			      B_2=\{y_j:j\in\mathbb Z_q\},\quad
			      B_3=\{z_k:k\in\mathbb Z_q\}.
		      \]
		      The triangles
		      \[	
			      x_i y_j z_{i+j}
			      \quad (i,j\in\mathbb Z_q)
		      \]
		      are edge-disjoint. Thus we get at least \(q^2\ge c_\beta^2n^2/2\)
		      edge-disjoint triangles inside \(B\).
		      
		      For each such triangle, its three edges have colors in \(C\). They cannot receive
		      three distinct colors from \(C\), because together with the three distinct
		      external colors incident with \(u\), the four vertices would form a rainbow
		      \(K_4\). Therefore each such triangle imposes a fixed loss. For a total loss factor:
		      \[	
			      \left(1 - \frac{5\cdot 4 \cdot 3}{5^3}\right)^{\!q^2}
		      \]
		      
		\item One color from $[r]\setminus C$ and two colors from $C$, say colors $a$ and $b$, occur on at least \(c_\beta n\) edges each of \(uB\). Taking triples from the corresponding three vertex classes,
		      we again obtain at least \(c_\beta^2n^2/2\) edge-disjoint triangles inside
		      \(B\). Such a triangle cannot receive exactly the three colors of \(C\setminus
		      \{a,b\}\), one on each edge, because then the resulting \(K_4\) would be rainbow.
		      Thus this alternative also gives a fixed exponential loss factor:
		      \[	
			      \left(\frac{5^3-6}{5^3}\right)^{\! q^2}.
		      \]
	\end{itemize}
	
	In either case, the total loss is a factor less than
	\[	
		\exp\left(\frac{-6q^2}{5^3}\right) \le \exp\left(\frac{-6\,c_\beta^2\,n^2}{5^3 \cdot 2}\right) \le \exp\left(\frac{-3\,\beta^2n^2}{5^3\cdot 10^4\,r^2}\right) \le \exp\bigl( -10^{-9}\beta^2n^2 \bigr).
	\]
	
	Finally, because $\alpha(\ln(1/\alpha)+3) \le 10^{-9}\beta^2$, the loss factor dominates the number of choices of $S$, $u$ and colors not in $C$. Hence the number of colorings in this subcase is at most
	\[	
		\exp\left( -c_3 n^2 \right)5^N \le \frac{\eps}{4}5^N,
	\]
	for $n$ large.
	
	\noindent\textbf{Case 3.2.} Finally, in our last case, we count the star-structured colorings. As in Case~$2$, we start choosing $M$, then the colors of the edges inside $A$ (which determines $S\cap E(A)$), and that can be done in
	\[	
		\exp\bigl(2t\ln n + 8t^2\bigr) \text{ ways}.
	\]

	The choice of non-large vertices, the edges of $S$ that touch them and their colors can be chosen in at most
	\[	
		2^{2t}\sum_{s=0}^{2\beta t n} \binom{2tn}{s}(r-5)^s \le \exp(4\beta tn\ln(1/\beta)) \text{ ways}.
	\]

	The remaining vertices of $A$ are large. The choice of all pairs \((P_u,a_u)\) cost at
	most
	\[	
		\left(\binom{r}{2}\cdot 5\right)^{\!2t} \le 1380^{2t} \le \exp(15t).
	\]
	
	For each large vertex, $u$, at most \(\beta n/100\) exceptional edges from $uB$ have colors outside $P_u\cup \{a_u\}$. The choice of those edges and their colors cost at most
	\[	
		\left( \sum_{i=0}^{\beta n/100}\binom{|B|}{i}r^i \right)^{\!p} \le \exp(3p\beta n\ln(1/\beta)).
	\]
	
	Finally, every other edge that touch a large vertex $u$ has color in $P_u\cup \{a_u\}$. So there are only $3$ choices for them. Let $h$ be the number of those good edges. And all remaining edges to be colored are in $B$ or between non-large vertices of $A$ and $B$, therefore, have at most five choices of colors. Notice also that
	\[	
		h \ge p(|B|-\beta n/100) \ge  \frac{\gamma t}{2}\left(n-2t-\frac{\beta n}{100}\right) \ge \frac{\gamma t}{2}\frac{n}{4} = \frac{\gamma tn}{8}.
	\]
	
	Altogether, the total number of colorings in this case is at most:
	\begin{align*}	
		 & \exp\bigl(2t\ln n + 8t^2 + 4\beta tn\ln(1/\beta)+15t+3p\beta n\ln(1/\beta)\bigr)\,3^h\,5^{N-h}
	\end{align*}
	
	Finally,
	\[	
		\left(\frac35\right)^{h} = \left(1-\frac25\right)^{h}
		\le
		\exp\Bigl(\frac{-2}{5}h\Bigr) \le
		\exp\Bigl(-\frac{1}{20}\gamma tn\Bigr).
	\]
	
	By the constants's hierarchy and the range for $t$ in Case 3,
	\[	
		\beta\ln(1/\beta)\le10^{-4}\gamma \quad \text{and} \quad t < K\alpha n \le 10^{-4}\gamma n,
	\]
	together with the fact that $p \le 2t$, all the costs are dominated by $(3/5)^h$.
	
	Thus, for each fixed $t\ge 1$, the number of star-structured colorings is at most
	\[	
		\exp\Bigl(-\frac{1}{40}\gamma tn\Bigr)\,5^N.
	\]
	The total contribution of Case 3.2 is attained by summing over all possible values of $t$:
	\[	
		\sum_{t=1}^{\lfloor n/2\rfloor}\exp\Bigl({-\frac{1}{40}\gamma tn}\Bigr)5^N
		\le \frac{\eps}{4}\,5^N.
	\]
	
	\noindent\textbf{Final count:} it is clear that the four cases cover all colorings with \(1\le |S|\le\alpha N\). Therefore the
	number of such colorings is at most \(\eps 5^N\). Hence
	\[	
		|\mathcal F_C(\alpha)|\le (1+\eps)5^N.\qedhere
	\]
\end{proof}

\begin{remark} While in Lemma~\ref{lem:structural-below} the fact that $r\le 24$ is crucial for the lemma to be valid, in Lemma~\ref{lem:fine-count} the bound on $r$ is used mostly for convenience (since we did not need to use if for large values of $r$): we use that $r-5 \le 19 \le e^3$, $r\le e^4$ and $\binom{r}{2}5 \le 1380$ to simplify certain computations, for example. But some version of Lemma~\ref{lem:fine-count}, possibly with $\alpha$ depending on $r$, should also be valid for $r\ge 25$.
\end{remark}

\subsection{Combining previous lemmas}
Finally, we combine the two previous lemma, to get the upper bound in Theorem~\ref{thm:main}(\ref{item:main2}).

\begin{lemma}	
	\label{prop:below}
	If \(6\le r\le24\), then for every \(\eps>0\), there exists
	\(n_0(r,\eps)\) such that, for every \(n\ge n_0(r,\eps)\),
	\[	
		\rho_{r,4}(K_n)
		\le
		\left(\binom r5+\eps\right)5^N.
	\]
\end{lemma}

\begin{proof}	
	Fix \(\eps>0\). Let
	\[	
		\eps'=\frac{\eps}{2\binom r5}.
	\]
	Let $\alpha=\alpha_{\mathrm{fine}}$ be given by Lemma~\ref{lem:fine-count}. Let
	\[	
		\zeta_{\mathrm{str}}=\zeta_{\mathrm{str}}(r,\alpha),
		\quad
		\delta_{\mathrm{str}}=\delta_{\mathrm{str}}(r,\alpha)
	\]
	be given by Lemma~\ref{lem:structural-below}.
	
	Choose a container parameter \(\xi>0\) such that
	\[	
		\xi\le \delta_{\mathrm{str}},
	\]
	and, for all sufficiently large \(n\),
	\begin{equation}	
		\label{eq:poor-sum}
		\exp(\xi n^2)5^{N-\zeta_{\mathrm{str}}n^2}
		\le
		\frac{\eps}{2}5^N.
	\end{equation}
	
	Apply Theorem~\ref{thm:containers} with this \(\xi\) to obtain a family $\mathcal{C}$ of templates containers for $K_n$. Every template $T\in \mathcal{C}$, satisfy $\Rfour(T) \le \xi n^4 \le \delta_{\mathrm{str}}n^4$. Therefore, we can apply Lemma~\ref{lem:structural-below} to each of them. The sum of the number of colorings generated by all templates that satisfy condition (i) of Lemma~\ref{lem:structural-below}, which we call poor templates, is at most $\frac{\eps}{2}5^N$, by equation~\ref{eq:poor-sum}.
	
	Every other template $T$ satisfy condition (ii) of Lemma~\ref{lem:structural-below}: meaning that for each of them there is a set \(C = C(T) \in\binom{[r]}5\) such that
	\[	
		|\{e:T(e)=C\}|\ge(1-\alpha)N.
	\]
	Therefore every valid coloring generated by such a container belongs to $\mathcal F_C(\alpha)$. Thus the colorings generated by all non-poor containers are contained in
	\[	
		\bigcup_{C\in\binom{[r]}5}\mathcal F_C(\alpha).
	\]
	By Lemma~\ref{lem:fine-count}, for every fixed \(C\),
	\[	
		|\mathcal F_C(\alpha)|\le (1+\eps')5^N.
	\]
	Therefore
	\[	
		\rho_{r,4}(K_n)
		\le
		\frac{\eps}{2}5^N
		+
		\binom r5(1+\eps')5^N
		\le
		\left(\binom r5+\eps\right)5^N.\qedhere
	\]
\end{proof}

\section{Concluding remarks}

It is clear that both constructions in Section~\ref{sec:lower} generalize to produce many $r$-colorings that are rainbow-$K_k$-free. 
Given $n, k\ge 3$, define $s = \binom{k}{2}-1$ and $N = \binom{n}{2}$.

For $r\le s$, it is trivial to see that $\rho_{r,k}(K_n) = r^N$.

For $r \ge s$, in the first construction we consider all colorings of $K_n$ that use exactly $s$ colors from $[r]$. 
Using the estimate for Stirling numbers of the second kind, as we did for $k=4$ (Section~\ref{sec:lower}), we obtain, for fixed $k\ge3$, $r\ge s$, and $\eps>0$, and all sufficiently large $n$,
\[	
	\rho_{r,k}(K_n) \ge \left(\binom{r}{s} - \eps\right) s^{\binom{n}{2}}.
\] 

For our second construction, on the other hand, fix a partition of the vertex set of $K_n$ into $k-2$ parts, each of size $\lfloor n/(k-2) \rfloor$ or $\lceil n/(k-2) \rceil$. Select a color from $[r]$ and color all edges within each part with that color. Color all remaining edges arbitrarily. Since every set of $k$ vertices must have at least two pairs of vertices that lie in the same part, it follows that no $K_k$ can be rainbow. There are $t_{k-2}(n)$ edges whose endpoints are not in the same part. Therefore,
\[	
	\rho_{r,k}(K_n) \ge r^{1+t_{k-2}(n)} \ge r^{\left(1 - \frac{1}{k-2}\right)\binom{n}{2}}.
\]
This can be further improved by considering all possible partitions of the vertex set. 

For fixed $k \ge 4$ and fixed $r > s^{(k-2)/(k-3)}$, the second lower bound is larger than the first for every $n$ sufficiently large. It would be interesting to determine whether these bounds are asymptotically optimal and whether these are the only possible regimes, as in the case $k=4$ covered by Theorem~\ref{thm:main}. 

\subsection*{Acknowledgements}

A much weaker version of the upper bound in Theorem~\ref{thm:main}\ref{item:main3}, of the form $(5^{10}r)^{n^2/4}$, was proved for $r > 5^{10}$ in early 2025 by the authors in collaboration with Rodrigo Ribeiro, then a master's student supervised by Fabrício S. Benevides. It was published (in Portuguese) in Ribeiro's master's thesis \cite{Ribeiro2025dissertation}.

\subsection*{Declaration of generative AI and AI-assisted technologies in the manuscript preparation process} 

During the preparation of this work, the authors used ChatGPT 5.6 and Gemini 3.1 in order to quickly obtain a proof for Lemma~\ref{lem:k4minus-codegrees} and to obtain an initial suggestion on how to organize the cases in Lemma~\ref{lem:fine-count}. After using these tools, the authors reviewed and edited the content as needed and take full responsibility for the content of the published article.

\bibliography{referencias.bib}
\bibliographystyle{amsplain}

\end{document}